\documentclass[12pt,a4paper]{amsart}
\usepackage{a4wide}
\usepackage{amsmath, amssymb, amsfonts,enumerate}
\usepackage[all]{xy}
\usepackage{amscd}
\usepackage{comment}
\usepackage{mathtools}
\usepackage{tikz}

\DeclarePairedDelimiter\floor{\lfloor}{\rfloor}

\newtheorem{theorem}{Theorem}[section]
\newtheorem{lemma}[theorem]{Lemma}
\newtheorem{corollary}[theorem]{Corollary}
\newtheorem{proposition}[theorem]{Proposition}

 \theoremstyle{definition}
 \newtheorem{definition}[theorem]{Definition}
 \newtheorem{remark}[theorem]{Remark}

 \newtheorem{example}[theorem]{Example}

\numberwithin{equation}{section}
\newcommand {\N}{\mathbb{N}} 
\newcommand {\Z}{\mathbb{Z}} 
\newcommand {\R}{\mathbb{R}} 

\newcommand{\T}{\mathbb{T}}

\newcommand{\FF}{\mathcal{F}}

\DeclareMathOperator{\M}{Mat}

\DeclareMathOperator{\Sym}{Sym}

\DeclareMathOperator{\Map}{Map}

\newcommand{\htopol}{{\text{\rm h}}_{\text{\rm top}}}

\newcommand{\cF}{{\mathcal F}}

\newcommand{\Nb}{{\mathbb N}}

\newcommand{\diam}{{\rm diam}}

\newcommand{\ev}{{\rm ev}}

\allowdisplaybreaks

\begin{document}
\title
[Expansive actions, sofic groups, and surjunctivity]{Expansive actions with specification of sofic groups, strong topological Markov property,
and surjunctivity}
\author{Tullio Ceccherini-Silberstein}
\address{Dipartimento di Ingegneria, Universit\`a del Sannio, 82100 Benevento, Italy}
\email{tullio.cs@sbai.uniroma1.it}
\author{Michel Coornaert}
\address{Universit\'e de Strasbourg, CNRS, IRMA UMR 7501, F-67000 Strasbourg, France}
\email{michel.coornaert@math.unistra.fr}
\author{Hanfeng Li}
\address{Department of Mathematics, SUNY at Buffalo, Buffalo, NY 14260-2900, USA}
\email{hfli@math.buffalo.edu}


\begin{abstract}
A dynamical system is a pair $(X,G)$, where $X$ is a compact metrizable
space and $G$ is a countable group acting by homeomorphisms of $X$. An endomorphism
of $(X,G)$ is a continuous selfmap of $X$ which commutes with the action of $G$.
One says that a dynamical system $(X,G)$ is surjunctive provided that every injective
endomorphism of $(X,G)$ is surjective (and therefore is a homeomorphism).
We show that when $G$ is sofic, every expansive dynamical system $(X,G)$
with nonnegative sofic topological entropy and satisfying the weak specification
and the strong topological Markov properties, is surjunctive.
\end{abstract}
\date{\today}
\maketitle

\tableofcontents

\section{Introduction}
Let $A$ be a finite set, called the \emph{alphabet}, and let $G$ be a group.
We equip $A^G = \prod_{g \in G} A$ with the \emph{prodiscrete} topology, i.e., the product topology obtained by taking the discrete topology on each factor $A$ of $A^G$, and the \emph{shift} action of $G$, given by $gx(h) \coloneqq  x(g^{-1}h)$ for all
$g,h \in G$ and $x \in A^G$. This action is clearly continuous.
A closed $G$-invariant subset $X$ of $A^G$ is called a \emph{subshift} and a continuous map $f \colon X \to X$ which commutes with the
shift is called a \emph{cellular automaton} (cf.\ \cite{book}).
\par
A group $G$ is called \emph{surjunctive} if, given any finite alphabet $A$,  every injective cellular automaton $f \colon A^G \to A^G$ is surjective (and therefore a homeomorphism) (see \cite[Chapter 3]{book}). Lawton~\cite{lawton} proved that all residually finite groups are surjunctive and Gottschalk~\cite{gottschalk} raised the question of determining those groups which are surjunctive. The statement asserting that every group is surjunctive is now commonly known as the \emph{Gottschalk conjecture}. The fact that all amenable groups are surjunctive follows from one of the implications of the Garden of Eden theorem established in~\cite{ceccherini} (see also \cite[Chapter 5]{book}).
Up to now, the largest class of groups that are known to be surjunctive is the class of \emph{sofic} groups introduced by
Gromov in~\cite{gromov-esav} and named by Weiss in~\cite{weiss-sgds} (see Section \ref{s:sofic}): surjunctivity of sofic groups was established by Gromov~\cite{gromov-esav} and Weiss~\cite{weiss-sgds} (see also~\cite{capraro-lupini}, \cite{book}, \cite{kerr-li} for expositions of this result).
The class of sofic groups contains all residually finite groups, all amenable groups, and, more generally, all residually amenable groups.
Actually, no example of a non-sofic group has yet been found.
\par
A \emph{dynamical system} is a pair $(X,G)$ consisting of a compact metrizable space $X$ equipped with a continuous action of a countable group $G$.
An \emph{endomorphism} of a dynamical system $(X,G)$ is a continuous $G$-equivariant map $f \colon X \to X$.
\par
For instance, if $A$ is a finite set, $G$ is a countable group, and $X \subset   A^G$ is a subshift, then $(X,G)$, where $X$ is equipped with the topology induced by the prodiscrete topology and the action of $G$ obtained by restriction of the shift action, is a dynamical system whose endomorphisms are the cellular automata $f \colon X \to X$.
\par
One says that a dynamical system $(X,G)$ is \emph{surjunctive} if every injective endomorphism of $(X,G)$ is surjective (and hence a homeomorphism). Also $(X,G)$ is \emph{expansive} if, given a compatible metric $\rho$ on $X$, there exists a constant $c > 0$ such that for all distinct $x,y \in X$ there exists $g \in G$ such that $\rho(gx,gy) > c$. By  compactness of $X$, the fact that the dynamical system
$(X,G)$ is expansive or not does not depend on the choice of the compatible  metric $\rho$. For instance, if $X \subset   A^G$ is a subshift, then $(X,G)$ is expansive.
See Section \ref{s:exp}.
\par
A \emph{sofic approximation} of a countable group $G$ is a sequence $\Sigma = (d_j, \sigma_j)_{j \in \N}$, where $d_j$ is a positive integer and $\sigma_j \colon G \to \Sym(d_j)$ is a map from $G$ into the symmetric group of degree $d_j$, satisfying the following conditions:

(i) $\lim_{j \to \infty} d_j = + \infty$,

(ii) $\lim_{j \to \infty} \eta_{d_j}(\sigma_j(s)\sigma_j(t), \sigma_j(st)) = 0$ for all $s,t \in G$,
and

(iii) $\lim_{j \to \infty} \eta_{d_j}(\sigma_j(s), \sigma_j(t)) = 1$ for all distinct $s,t \in G$,

\noindent
where $\eta_{d_j}$ denotes the normalized \emph{Hamming distance} on $\Sym(d_j)$.
A countable group $G$ is \emph{sofic} if it admits a sofic approximation. More generally, a group is \emph{sofic} if all of its
finitely generated subgroups are sofic.
\par
For sofic groups, David Kerr and the third named author \cite{{kerr-li-0}, kerr-li, kerr-li-book}
defined and studied the \emph{topological entropy} $h_\Sigma(X,G)$ of a dynamical system $(X,G)$ relative to a sofic approximation $\Sigma$ of the group $G$. See Section \ref{s:sofic-top-entropy}.
\par
The notion of \emph{specification}, a strong orbit tracing property, was introduced
by Rufus Bowen for $\Z$-actions in relation to his studies on Axiom A diffeomorphisms in \cite{bowen-periodic-1971} (see
also \cite[Definition 21.1]{denker-grillenberger-sigmund}) and was subsequently extended to $\Z^d$-actions by Ruelle in
\cite{ruelle-statistical-1973}. Several versions and generalizations of specification have appeared in the literature
(see, in particular, \cite[Definition 5.1]{lind-schmidt} and \cite[Definition 6.1]{chung-li-2015}).
Here is the one we need (cf.~\cite[Definition 6.1]{chung-li-2015}, see also \cite{Li19}).
A dynamical system $(X,G)$ has the \emph{weak specification property} if for any $\varepsilon > 0$ there exists a
 finite symmetric subset $F \subset   G$ containing $1_G$ satisfying the following property: if $(\Omega_i)_{i \in I}$ is any finite family
of finite subsets of $G$ such that $F \Omega_i \cap \Omega_j = \varnothing$ for all distinct $i,j \in I$, and $(x_i)_{i \in I}$
is any family of points in $X$, then there exists $x \in X$ such that $\rho(sx,sx_i) \leq \varepsilon$ for all $i \in I$ and
$s \in \Omega_i$, where $\rho$ is any metric compatible with the topology on $X$.
When $X \subset   A^G$ is a subshift, the dynamical system $(X,G)$ has the weak specification property if and only if $X$ is
\emph{strongly irreducible} (cf.\ \cite{csc-myhill-monatsh}, \cite[Proposition A.1]{Li19}). See Section \ref{s:wsp}.
\par
In \cite{BGL}, Barbieri, Garc\'\i a-Ramos, and the third named author, inspired by investigations in the context of symbolic dynamics
\cite{BGMT}, introduced and studied various notions of topological Markov properties. Among them, the following will play a crucial role in our work.  A dynamical system  $(X,G)$ satisfies the \emph{strong topological Markov property} (briefly, \emph{sTMP}) if for any $\varepsilon > 0$ there exists $\delta > 0$ and a finite subset $F \subset   G$ containing $1_G$ such that the following holds:
given any finite subset $A \subset   G$ and $x,y \in X$ such that $\rho(sx,sy) \leq \delta$ for all $s \in FA \setminus A$,
there exists $z \in X$ satisfying that $\rho(sz,sx) \leq \varepsilon$ for all $s \in FA$ and $\rho(sz,sy) \leq \varepsilon$ for all
$s \in G \setminus A$.
When $X \subset   A^G$ is a subshift, the strong topological Markov property for $(X,G)$ is equivalent to $X$ being \emph{splicable}
in the sense of Gromov \cite[Section 8.C]{gromov-esav} (see also \cite[Section 3.8]{CFS}).
For instance, every subshift \emph{of finite type} is splicable. See Section \ref{s:stmp}.
\par
We are now in a position to state the main result of the paper.

\begin{theorem}
\label{t:main}
Let $G$ be a countable sofic group and let $(X,G)$ be an expansive dynamical system with the weak specification property and the strong topological Markov property. Assume that $h_\Sigma(X,G) \geq 0$ for some sofic approximation $\Sigma$ of $G$. Then
$(X,G)$ is surjunctive.
\end{theorem}

The following result, which is of independent interest, is our key ingredient for proving Theorem~\ref{t:main}:

\begin{proposition}
\label{p:main}
Let $G$ be a countable sofic group and let $(X,G)$ be an expansive dynamical system with the weak specification property and the strong topological Markov property. Assume that $h_\Sigma(X,G) \geq 0$ for some sofic approximation $\Sigma$ of $G$.
Suppose that $Y \subset X$ is a proper closed $G$-invariant subset. Then $h_\Sigma(Y,G) < h_\Sigma(X,G)$.
\end{proposition}

As an immediate consequence of Theorem \ref{t:main} we deduce the following:

\begin{corollary}
\label{c:one}
Let $A$ be a finite set and let $G$ be a countable sofic group.
Suppose that $X \subset A^G$ is a splicable (e.g., of finite type) strongly irreducible subshift.
Assume that $h_\Sigma(X,G) \geq 0$ for some sofic approximation $\Sigma$ of $G$.
Then $(X,G)$ is surjunctive. \qed
\end{corollary}

As the full shift $A^G$ is (trivially) splicable (in fact of finite type) and strongly irreducible, 
and $h_\Sigma(A^G,G) = \log |A| \geq 0$ for all sofic approximations $\Sigma$ of any sofic group $G$,
from Corollary \ref{c:one}
one immediately recovers the Gromov-Weiss surjunctivity theorem for sofic groups.
When $G$ is an amenable group, the surjunctivity result in Corollary \ref{c:one} was proved under the additional assumption of $X$ being of finite type by Fiorenzi \cite[Corollary 4.8]{fiorenzi-strongly}; 
the first two named authors \cite{csc-myhill-monatsh} then showed that the strong irreduciblity of the subshift is sufficient to imply surjunctivity when $G$ is amenable. Note that, however, the strong irreducibility condition cannot be dropped: for instance, Weiss \cite{weiss-sgds} gave an example of a subshift of finite type over $G = \Z$ which is not surjunctive (see Example \ref{ex:weiss}) as well as an example of a topologically mixing subshift of finite type over  $G = \Z^2$ which is not surjunctive either.

An \emph{algebraic dynamical system} is a dynamical system $(X,G)$, where $X$ is a compact metrizable abelian group and $G$ is a countable
group acting by continuous automorphisms of $X$. Given an algebraic dynamical system $(X,G)$, the Pontryagin dual $\widehat{X}$, when
equipped with the dual action of $G$, is a countable left $\Z[G]$-module (see Section \ref{s:algebraic-actions}).
In fact, Pontryagin duality establishes a bijective correspondence between algebraic dynamical systems over $G$
and countable left $\Z[G]$-modules. 

An algebraic dynamical system $(X,G)$ is said to be \emph{finitely presented} provided its
Pontryagin dual $\widehat{X}$ is a finitely presented left $\Z[G]$-module.
This amounts to saying that there exists an $n \times m$ matrix $A$ with coefficients in $\Z[G]$ such that
$\widehat{X} = \Z[G]^m/(\Z[G]^n A)$; we then denote $X$ by $X_A$.
In the particular case when $A$ is a square matrix,
the associated algebraic dynamical system $(X_A,G)$ is called a \emph{generalized principal algebraic dynamical system}.
Note that if $(X,G)$ is an algebraic dynamical system, then $G$ fixes the identity element of $X$.
Therefore, if in addition the group $G$ is sofic, then $(X,G)$ satisfies $h_\Sigma(X,G) \geq 0$ for all 
sofic approximations $\Sigma$ of $G$.
From Theorem \ref{t:main}, in combination with results of Meyerovitch \cite{Meyerovitch}, Ren \cite{Ren}, and
Barbieri, Garc\'\i a-Ramos and the third named author \cite{BGL}, we then deduce the following:

\begin{corollary}
\label{c:alg-finitely-presented}
Let $G$ be a countable sofic group. Let $A \in \M_n(\Z[G])$ and suppose that $A$ is invertible in $\M_n(\ell^1(G))$.
Then the generalized principal algebraic dynamical system $(X_A,G)$ is surjunctive.
\end{corollary}

Here $\ell^1(G) \coloneqq \{f \in \R^G: \sum_{g \in G} |f(g)| < \infty\}$ is the Banach algebra of all real summable functions on $G$.
Note that when $n=1$ and $m \geq 2$ is an integer, then taking $f = m {1_G} \in \Z[G]$ one has $X_f = \{1,2, \ldots, m\}^G$, the full shift over an $m$-elements alphabet set. From Corollary \ref{c:alg-finitely-presented} one then recovers, once more, the Gromov-Weiss surjunctivity theorem for sofic groups.

We don't know if the above corollary can be extended to all expansive finitely presented algebraic dynamical systems.

Surjunctivity of algebraic dynamical systems, in relation with the so-called topological rigidity property, was investigated also in \cite{BCC}.

We conclude with a surjunctivity result for actions of amenable groups on (not necessarily abelian) compact metrizable groups by continuous automorphisms. First, recall that given a countable amenable group $G$, a dynamical system $(X,G)$ has \emph{completely positive entropy} (briefly \emph{CPE}) if every nontrivial factor has positive topological entropy. We then have the following:

\begin{theorem} \label{T-CPE to surjunctive}
Let $G$ be a countably infinite amenable group $G$ and let $(X,G)$ be a dynamical system,
where $X$ is a compact metrizable group and $G$ acts by continuous automorphisms of $X$. 
Suppose that $(X,G)$ has CPE and that $\htopol(X, G) < \infty$ (e.g., $(X,G)$ is expansive). 
Then $(X,G)$ is surjunctive. 
\end{theorem}

\noindent
{\bf Acknowledgments.} We express our gratitude to Felipe Garc\'ia-Ramos for several interesting remarks. 
H.\ Li wishes to thank Filippo Tolli (Universit\`a Roma Tre) and the Department SBAI of the
Sapienza University in Rome for hospitality during a visit in December 2019, in which this work started. He was partially supported by NSF grant~DMS-1900746.\\


\section{Preliminaries}
\label{sec:background}
In this section we fix the notation and review some background material on dynamical systems, sofic groups, and sofic entropy.

\subsection{General notation}
Throughout the paper, we use $| \cdot |$ to denote cardinality of finite sets.
\par
We write $\N$ for the set of non-negative integers.
\par
Given sets $A$ and $B$, we denote by $A^B$ the set consisting of all maps $x \colon B \to A$.
We will sometimes regard an element $x \in A^B$ as a family of elements of $A$ indexed by $B$, namely $x = (x(b))_{b \in B}$,
where $x_b \coloneqq  x(b)$ for all $b \in B$.
If $B' \subset   B$ and $x \in A^B$, we denote by $x\vert_{B'} \in A^{B'}$ the \emph{restriction} of $x$ to $B'$, that is, the map
$x\vert_{B'} \colon B' \to A$ defined by $x\vert_{B'}(b') \coloneqq x(b')$ for all $b' \in B'$.
\par
Given a set $X$, we denote by $\Sym(X)$ the \emph{symmetric group} of $X$,
i.e., the group consisting of all bijective maps $\sigma \colon X \to X$ with the composition of maps as the group operation.
\par
Given a positive integer $d$, we write $[d] \coloneqq \{1,2, \dots, d\}$. To simplify notation, we write $\Sym(d)$ instead of
$\Sym([d])$.
\par
Given a group $G$, we denote by $\cF(G)$ the set of all nonempty finite subsets $F \subset   G$. Also,
given $F \in \cF(G)$, we say that a (not necessarily finite) subset $V \subset   G$ is
$F$-\emph{separated} if $Fs \cap Ft = \varnothing$ for all distinct $s,t \in V$.
\par
Given a metric space $(X,\rho)$ and a subset $Y \subset   X$, we denote by $\diam(Y,\rho) \coloneqq \sup\{\rho(y_1,y_2): y_1, y_2 \in Y\}$
the \emph{diameter} of $Y$. For $x \in X$, we write $\rho(x,Y) \coloneqq \inf_{y \in Y} \rho(x,y)$.
Moreover, given $\varepsilon > 0$ one says that $Y$ is $(\rho, \varepsilon)$-\emph{separated} provided
that $\rho(y,y') \geq \varepsilon$ for all distinct $y,y' \in Y$.

\subsection{Actions}
An \emph{action} of a group $G$ on a set $X$ is a group morphism $\alpha \colon G \to \Sym(X)$.
If $\alpha$ is an action of a group $G$ on a set $X$, given $g \in G$ and $x \in X$,
we usually write $\alpha_g(x)$ or, if the action is clear from the context, $g x$, instead of $\alpha(g)(x)$.
\par
An action of a group $G$ on a topological space $X$ is said to be continuous if its image is contained in the group of
homeomorphisms of $X$.
\par

\subsection{Shifts and subshifts}
\label{s:shifts}
Let $A$ be a finite set and let $G$ be a countable group.
We equip $A$ with the discrete topology and $A^G = \prod_{g \in G} A$ with the product topology (this is also called
the \emph{prodiscrete topology} on $A^G$).
Note that $A^G$ is a compact space (by Tychonoff's theorem) and is metrizable since $G$ is countable. A metric inducing
the prodiscrete topology can be constructed as follows.
Let $(\Omega_n)_{n \in \N}$ be a sequence of finite subsets of $G$ satisfying that
$\Omega_0 = \varnothing$, $\Omega_1 = \{1_G\}$, $\Omega_n^{-1} \coloneqq \{g^{-1}: g \in \Omega_n\} = \Omega_n$,
$\Omega_n \subset   \Omega_{n+1}$ for all $n \in \N$, and $\bigcup_{n \in \N} \Omega_n = G$.
Then the map $\rho \colon A^G \times A^G \to [0,1]$ defined by setting
\begin{equation}
\label{e:metric-A-G}
\rho(x,y) \coloneqq 2^{-n(x,y)}, \mbox{ where $n(x,y) \coloneqq \sup \{n \in \N: x\vert_{\Omega_n} = y\vert_{\Omega_n}\}$}
\end{equation}
(with the usual convention that $2^{-\infty} = 0$) for all $x, y \in A^G$, is a metric on $A^G$ inducing the prodiscrete topology.
\par
We  make $G$ act on $A^G$ by the \emph{shift} action, given by the map $G \times A^G \to A^G$, $(g,x) \mapsto g x$, where
\begin{equation}
\label{e:shift-action}
gx(h) \coloneqq  x(g^{-1}h) \text{   for all $g,h \in G$ and $x \in A^G$.}
\end{equation}
The action of $G$ on $A^G$ is clearly continuous.
\par
A closed $G$-invariant subset of $A^G$ is called a \emph{subshift}.
\par
A subshift $X \subset   A^G$ is said to be \emph{of finite type} provided there exists a finite subset $\Omega \subset   G$ and ${\mathcal A} \subset   A^\Omega$ such that $X = \{x \in A^G: (gx)\vert_\Omega \in {\mathcal A} \mbox{ for all } g \in G\}$ (such a set ${\mathcal A}$ is then called a \emph{defining set of admissible patterns} for $X$ and the subset $\Omega$ is called a \emph{memory set} for $X$).

\subsection{Dynamical systems}
A \emph{dynamical system} is a pair $(X,G)$ consisting of a compact metrizable space $X$, called the \emph{phase space}, equipped with a continuous action of a countable group $G$.
\par
Let $(X,G)$ be a dynamical system and let $\rho$ be a metric on $X$ compatible with the topology.
Given a nonempty subset $A \subset   G$, we denote by $\rho_A$ the metric on $X$ defined by setting
\begin{equation}
\label{e:rho-A}
\rho_A(x,y) = \sup_{g \in A} \rho(gx,gy)
\end{equation}
for all $x,y \in X$. Note that $\rho_{\{1_G\}} = \rho$.

\subsection{Expansivity}
\label{s:exp}

\begin{definition}
\label{def:expansive}
The dynamical system $(X,G)$ is \emph{expansive} if there exists a constant $c > 0$ such that for all distinct $x,y \in X$ there exists $g \in G$ such that $\rho(gx,gy) > c$.  Such a value $c$ is called an \emph{expansivity constant} for $(X,G)$ (relative to the metric $\rho$).
\end{definition}

Compactness of the phase space $X$ guarantees that the definition of expansivity does not depend on the choice of the metric $\rho$.
Note that $c > 0$ is an expansivity constant for $(X,G)$ if and only if $\rho_G(x,y) > c$ for all distinct $x,y \in X$ (cf.\ \eqref{e:rho-A}).
For example, if $A$ is a finite set and $G$ is a countable group, then the symbolic dynamical system $(A^G,G)$ is expansive (with
expansivity constant $c = 1/2$ relative to the metric $\rho$ defined by~\eqref{e:metric-A-G}).

The following lemma (cf.\ \cite[Lemma 3.13]{BGL}) was proved by Bryant \cite[Theorem 5]{Bryant} (cf.\ \cite[Proposition 2]{Achigar})
for $G = \Z$.

\begin{lemma}[Uniform expansivity]
\label{l:uniform-exp}
Let $(X,G)$ be a dynamical system, let $\rho$ be a compatible metric on $X$, and let $c > 0$.
The following conditions are equivalent:
\begin{enumerate}[{\rm (a)}]
\item
$(X,G)$ is expansive and $c$ is an expansivity constant for $(X,G)$ relative to $\rho$;
\item for every $\varepsilon > 0$, there exists a nonempty finite subset $K \subset G$ such that the following holds: if $x, y \in X$ satisfy
that $\rho_K(x,y) \leq c$, then $\rho(x,y) < \varepsilon$.
\end{enumerate}
\end{lemma}
\begin{proof}
Assume (a) and suppose, by contradiction, that (b) fails to hold, that is, there exists $\varepsilon_0 > 0$ such that for every
nonempty finite subset $K \subset G$ there exist $x_K, y_K \in X$ such that $\rho_K(x_K,y_K) \leq c$ but $\rho(x_K,y_K) \geq \varepsilon_0$.
Consider the nets $(x_K)_{K \in \FF(G)}$ and $(y_K)_{K \in \FF(G)}$ of points in $X$, where the common index set $\FF(G)$ is partially ordered by inclusion. Since $X$ is compact, we can find a cluster point $(x,y)$ in $X \times X$ for $(x_K, y_K)_{K \in \FF(G)}$.
Note that, by continuity, $\rho(x,y) \geq \varepsilon_0$, so that, in particular, $x \neq y$.
Let $g \in G$. If $K \in \FF(G)$ contains $g$ one has $\rho(gx_K,gy_K) \leq c$. By continuity, we have $\rho(gx, gy) \leq c$.
As $g$ was arbitrary, we deduce that $\rho_G(x,y) \leq c$.
Since $x \neq y$, this contradicts the assumption that $c$ is an expansivity constant.
This proves the implication (a) $\Rightarrow$ (b).
\par
Conversely, suppose (b). Let $x,y \in X$ such that $\rho_G(x,y) \leq c$ and let us show that $x = y$.
Let $\varepsilon > 0$ and  let $K \subset   G$ be a  finite subset as in (b).
As $\rho_K(x,y) \leq \rho_G(x,y)\leq c$, we have $\rho(x,y) < \varepsilon$.
Since $\varepsilon$ was arbitrary, this implies that $\rho(x,y) = 0$, that is, $x = y$.
This shows that $c$ is an expansivity constant for $(X,G)$.
\end{proof}

\subsection{The strong topological Markov property}
\label{s:stmp}
\begin{definition}
\label{def:sTMP}
The dynamical system $(X,G)$ satisfies the \emph{strong topological Markov property} (briefly, \emph{sTMP}) if
for any $\varepsilon > 0$ there exists $\delta > 0$ and a finite subset $F \subset   G$ containing $1_G$ such that the following holds.
Given any finite subset $A \subset   G$ and $x,y \in X$ such that $\rho_{FA \setminus A}(x,y) \leq \delta$, there exists $z \in X$ satisfying
that $\rho_{FA}(z,x) \leq \varepsilon$ and $\rho_{G \setminus A}(z,y) \leq \varepsilon$.
\end{definition}

\begin{definition}
\label{def:usTMP}
The dynamical system $(X,G)$ satisfies the \emph{uniform strong topological Markov property} if
for any $\varepsilon > 0$ there exists $\delta > 0$ and a finite subset $F \subset   G$ containing $1_G$ such that the following holds.
Given a finite subset $A \subset G$, an $FA$-separated subset $V \subset G$, a family $(x_v)_{v \in V}$ of points in $X$, and $y \in X$
such that $\rho_{(FA \setminus A)s} (x_s,y) \leq \delta$ for all $s \in V$, there exists $z \in X$ satisfying that
$\rho_{FAs}(x_s,z) \leq \varepsilon$ for all $s \in V$ and $\rho_{G \setminus AV}(y,z) \leq \varepsilon$.
\end{definition}

\begin{remark}
\label{r:exp-sTMP-usTMP}
Clearly, every dynamical system with uniform sTMP satisfies sTMP (just take $V = \{1_G\}$).
Conversely, it is shown in \cite[Remark 3.14]{BGL} that every expansive dynamical system satisfying sTMP satisfies the uniform sTMP.
\end{remark}

In the symbolic setting, we have the following characterization of sTMP and usTMP
(see \cite[Corollary 3.8]{BGL} and \cite[Proposition 3.16]{CFS}).

\begin{proposition}
\label{p:sTMP}
Let $A$ be a finite set, let $G$ be a countable group, and let $X \subset A^G$ be a subshift.
Then the following conditions are equivalent:
\begin{enumerate}[{\rm (a)}]
\item $(X,G)$ satisfies the strong topological Markov property;
\item $(X,G)$ satisfies the uniform strong topological Markov property;
\item there exists a finite subset $\Delta \subset G$ containing $1_G$ such that the following holds. Given any finite subset $A \subset   G$ and $x,y \in X$ such that $x\vert_{A\Delta \setminus A} = y\vert_{A\Delta \setminus A}$, the configuration $z \in A^G$ such that $z\vert_A = x\vert_A$ and $z\vert_{G \setminus A} = y\vert_{G \setminus A}$ belongs to $X$.
\end{enumerate}
Moreover, a sufficient condition for the above equivalent conditions is $X$ being of finite type. 
\qed
\end{proposition}

\begin{remark}
\label{r:splicable}
A subshift $X \subset A^G$ satisfying condition (c) (equivalently, conditions (a) and (b)) in Proposition \ref{p:sTMP} was termed
\emph{splicable} by Gromov in \cite[Section 8.C]{gromov-esav} (see also \cite[Section 3.8]{CFS}).
\end{remark}

\subsection{The weak specification property}
\label{s:wsp}
\begin{definition}
\label{def:wSP}
The dynamical system $(X,G)$ satisfies the \emph{weak specification property} (briefly, \emph{wSP}) if for any $\varepsilon > 0$
there exist a finite symmetric subset $F\subset   G$ containing $1_G$ such that the following holds. Given finite subsets $A_1, A_2, \ldots, A_n \subset G$ satisfying $FA_i\cap FA_j =\varnothing$ for all distinct $i,j=1,2,\ldots,n$, and $x_1, x_2, \ldots, x_n \in X$,
there exists $z \in X$ such that $\rho_{A_i}(x_i, z)\leq \varepsilon$ for all $i=1,2,\ldots,n$.
\end{definition}

\begin{remark}
\label{r:wSP-infinite}
If wSP holds, then one may allow the subsets $A_1, A_2, \ldots, A_n \subset   G$ in Definition \ref{def:wSP} to be infinite.
\end{remark}

In the symbolic setting, the wSP has the following characterization (see \cite[Proposition A.1]{Li19} and 
\cite[Proposition 6.7]{csc:exp-spec-top-ent-Myhill}).

\begin{proposition}
\label{p:wSP}
Let $A$ be a finite set, let $G$ be a countable group, and let $X \subset   A^G$ be a subshift.
Then the following conditions are equivalent:
\begin{enumerate}[{\rm (a)}]
\item the dynamical system $(X,G)$ satisfies
the weak specification property;
\item there exists a finite subset $\Delta \subset   G$ containing $1_G$ such that the following holds.
Given finite subsets $A_1, A_2 \subset   G$ such that $A_1 \Delta \cap A_2 = \varnothing$ and $x_1, x_2 \in X$, there
exists $z \in X$ such that $z\vert_{A_1} = x_1\vert_{A_1}$ and $z\vert_{A_2} = x_2\vert_{A_2}$.
\end{enumerate} \qed
\end{proposition}

\begin{remark}
A subshift satisfying condition (b) (equivalently, condition (a)) in Proposition \ref{p:wSP} is called \emph{strongly irreducible} 
(cf.\ \cite{fiorenzi-strongly}, \cite{csc-myhill-monatsh}).
\end{remark}

We mention that M.~Doucha \cite{doucha} recently showed that any expansive action of a countable amenable group $G$ on a compact metrizable space $X$ satisfying both the weak specification and the strong topological Markov properties also satisfies the so-called \emph{Moore property}, i.e.\ every surjective endomorphism of $(X,G)$ is {\it pre-injective} 
(the latter means that the restriction to each homoclinicity class of $(X,G)$ is injective). 
This, together with an earlier result of the third-named author \cite{Li19} (where the strong topological Markov property is not needed),  namely, the so-called \emph{Myhill property}, i.e.\ every pre-injective endomorphism of $(X,G)$ is surjective, establishes the Garden of Eden theorem for all expansive actions of countable amenable groups on compact metrizable spaces satisfying the weak specification and the strong topological Markov properties.

\subsection{The pseudo-orbit tracing property}
\label{s:POTP}
The pseudo-orbit tracing property (briefly, POTP) was originally introduced and studied by
R.~Bowen \cite{bowen-Axiom-A} for $\Z$-actions, motivated by the study of Axiom A diffeomorphisms.
Walters \cite{walters} studied the relationships between POTP and topological stability and proved, among other things,
that every topologically stable homeomorphism of a compact manifold of dimension $\geq 2$ has the POTP.
Chung and Lee \cite{chung-lee} recently considered the pseudo-orbit tracing property for actions of
(finitely generated) countable groups and extended Walters's results. In particular, they proved that
every expansive dynamical system $(X,G)$ with the POTP is topologically stable.
Here is the definition:

\begin{definition}
\label{def:POTP}
Let $(X,G)$ be a dynamical system and let $\rho$ be a compatible metric on $X$.
Given $\delta > 0$ and a finite subset $S \subset   G$ an $(S,\delta)$ \emph{pseudo-orbit} is
a sequence $(x_g)_{g\in G}$ in $X$ such that $\rho(sx_g, x_{sg}) < \delta$
for all $s \in S$ and $g \in G$. Then one says that $(X,G)$ has the \emph{pseudo-orbit tracing property} if
for every $\varepsilon > 0$ there exists $\delta > 0$ and a finite subset $S \subset   G$ such that for
every $(S,\delta)$ pseudo-orbit $(x_g)_{g\in G}$ in $X$ there exists $x \in X$ such that
$\rho(gx, x_g) < \varepsilon$ for all $g \in G$
(one then says that the pseudo-orbit is $\varepsilon$-\emph{traced} by $x$).
\end{definition}

Barbieri, Garc\'\i a-Ramos, and the third-named author \cite[Proposition 3.10]{BGL}
proved that every dynamical system with the pseudo-orbit tracing property has the strong topological Markov property.

\subsection{Algebraic dynamical systems}
\label{s:algebraic-actions}
Let $L$ be a locally compact Abelian group. The \emph{Pontryagin dual} of $L$ is the set $\widehat{L}$ of all
continuous group homomorphism $\chi \colon L \to \T = \R/\Z$, called the \emph{characters} of $L$. 
When equipped with the topology  given by uniform convergence on compact sets, $\widehat{L}$ is a locally compact Abelian group
as well. Moreover, if $L$ is compact (resp.\ discrete, resp.\ countable and discrete) then $\widehat{L}$ is discrete 
(resp.\ compact, resp.\ compact metrizable).
The so-called \emph{Pontryagin duality} establishes that the \emph{evaluation map} $\ev_L \colon L \to \widehat{\widehat{L}}$, 
defined by setting $\ev_L(\ell)(\chi) \coloneqq \chi(\ell)$
for all $\ell \in L$ and $\chi \in \widehat{L}$, yields a canonical isomorphism (of locally compact topological groups) of $L$
into its bi-dual $\widehat{\widehat{L}}$. (We refer to \cite{morris} for more on Pontryagin duality.)
\par
Let $G$ be a countable group and let $\Z[G]$ denote the integral group ring of $G$.
Suppose that $G$ acts on $L$ by continuous group automorphisms. 
Then the action of $G$ induces an action on the Pontryagin dual $\widehat{L}$ by continuous group automorphisms and, by linearity,
a left $\Z[G]$-module structure on both $L$ and $\widehat{L}$.
In particular, the evaluation map is $G$-equivariant.
\par
An \emph{algebraic dynamical system} is a pair $(X,G)$, where $X$ is a compact metrizable Abelian topological group
and $G$ is a countable group acting on $X$ by continuous group automorphisms.
\par
Let $M$ be a countable left $\Z[G]$-module. Then, if we equip the Abelian group $M$ with its discrete topology, 
its Pontryagin dual $\widehat{M}$ is a compact metrizable Abelian group and $(\widehat{M},G)$ is an algebraic dynamical system.
In this way, algebraic dynamical systems with acting group $G$ are in one-to-one correspondence with countable left $\Z[G]$-modules.
\par
An algebraic dynamical system $(X,G)$ is said to be \emph{finitely presented} provided its
Pontryagin dual $\widehat{X}$ is a finitely presented left $\Z[G]$-module.
For a ring $R$ we denote by $\M_{n,k}(R)$ the left $R$-module of all $n \times k$ matrices with coefficients in $R$.
If $n = k$ we simply write $\M_n(R)$ instead of $\M_{n,k}(R)$.
Let $k, n \in \N$ and  let $A \in \M_{n,k}(\Z[G])$. Then $\Z[G]^n A$ is a finitely generated left $\Z[G]$-submodule of $\Z[G]^k$.
We denote by $M_A \coloneqq  \Z[G]^k/(\Z[G]^n A)$ the corresponding
finitely presented left $\Z[G]$-module and by $X_A \coloneqq \widehat{M_A}$ its Pontryagin dual.
The algebraic dynamical system $(X_A,G)$ is called the \emph{finitely presented algebraic dynamical system} associated with $A$.
An algebraic dynamical system $(X,G)$ is finitely presented if and only if there exist $k, n \in \N$ and $A \in \M_{n,k}(\Z[G])$
such that $(X,G)$ is \emph{isomorphic} to $(X_A,G)$ (i.e., there exists a continuous $G$-equivariant group isomorphism $X \to X_A$).
\par
An algebraic dynamical system $(X,G)$ is said to be \emph{principal} if $\widehat{X}$ is a cyclic left $\Z[G]$-module whose
annihilator is a principal left ideal of $\Z[G]$. This amounts to saying that $(X,G)$ is isomorphic to $(X_f,G)$ for some $f \in \Z[G]$,
where $(X_f,G)$ is the algebraic dynamical system whose Pontryagin dual is $\Z[G]/\Z[G]f$ (here $k = n = 1$).

\subsection{The Hamming distance}
Let $d$ be a positive integer.
The map $\eta_d \colon \Sym(d) \times \Sym(d) \to [0,1]$ defined by setting
\begin{equation}
\label{e:def-hamming}
\eta_d(\sigma,\sigma') \coloneqq  \frac{1}{d} \vert \{a \in [d]: \sigma(a) \not=  \sigma'(a) \} \vert
\end{equation}
for all $\sigma, \sigma' \in \Sym(d)$, is a bi-invariant metric on $\Sym(d)$, called the \emph{Hamming distance}.

\subsection{Sofic groups}
\label{s:sofic}
Let $G$ be a countable group.
A \emph{sofic approximation} of $G$ is a sequence $\Sigma = (d_j,\sigma_j)_{j \in \N}$,
such that the following conditions are satisfied:
\begin{itemize}
\item
$d_j$ is a positive integer and $\lim_{j \to \infty} d_j = + \infty$;
\item
$\sigma_j \colon G \to \Sym(d_j)$ is a map from $G$ into the symmetric group of $[d_j]$;
\item
$\lim_{j \to \infty} \eta_{d_j}(\sigma_j(s t), \sigma_j(s) \sigma_j(t)) = 0$ for all $s,t \in G$;
\item
$\lim_{j \to \infty} \eta_{d_j}(\sigma_j(s),\sigma_j(s)) = 1$ for all distinct $s, t \in G$,
\end{itemize}
where $\eta_{d_j}$ denotes the Hamming distance on $\Sym(d_j)$ defined by Formula~\eqref{e:def-hamming}.
\par
One says that a (possibly uncountable) group is \emph{sofic} if each of its finitely generated subgroups admits a sofic approximation.

\subsection{Sofic topological entropy}
\label{s:sofic-top-entropy}
In this section we recall the definition of the sofic topological entropy introduced by Kerr and the third-named author
in~\cite[Chapter~10]{kerr-li-book}.
\par
Let  $G$ be a countable sofic group and let $(X,G)$ be a dynamical system.
We fix a metric $\rho$ on $X$ compatible with the topology and a sofic approximation $\Sigma = (d_j,\sigma_j)_{j \in \N}$ of $G$.
\par
Let $d \geq 1$ be an integer. We regard $X^d$ as the set $\{\varphi \colon [d] \to X \}$ of all maps from $[d]$ to $X$.
We equip $X^d$  with the product topology and the diagonal action of $G$
given  by $(g \varphi)(a) \coloneqq g \varphi(a)$ for all $\varphi \in X^d$ and $a \in [d]$.
Define the metrics $\rho_2$ and $\rho_\infty$ on $X^d$ by setting
\begin{align*}
\rho_2(\varphi,\psi) &\coloneqq \left( \frac{1}{d} \sum_{a \in [d]}(\rho (\varphi(a),\psi(a)))^2 \right)^{1/2}  \text{  and} \\
\rho_\infty(\varphi,\psi) &\coloneqq \max_{a \in [d]} \rho(\varphi(a),\psi(a))
\end{align*}
for all $\varphi, \psi \in X^d$.
\par
Given a nonempty finite subset $F \subset   G$, a map $\sigma \colon G \to \Sym(d)$, and a real number $\delta > 0$, we denote by
$\Map(X,\rho,F,\delta,\sigma)$ the set
consisting  of all maps $\varphi \colon [d] \to X$ such that
\[
\rho_2(\varphi \circ \sigma_s, s  \varphi) < \delta \ \mbox{ for all } s \in F,
\]
equivalently,
\begin{equation}
\label{e:map-phi}
\sum_{a \in [d]} \rho(\varphi(\sigma_sa), s\varphi(a))^2 < d \delta^2 \ \mbox{ for all } s \in F.
\end{equation}
\par
Given $\varepsilon > 0$, we denote by $N_\varepsilon(\Map(X,\rho,F,\delta,\sigma), \rho_\infty)$ the maximal cardinality of
a $(\rho_\infty,\varepsilon)$-separated subset $Z \subset   \Map(X,\rho,F,\delta,\sigma)$. Note that
$N_\varepsilon(\Map(X,\rho,F,\delta,\sigma), \rho_\infty) = - \infty$ if $\Map(X,\rho,F,\delta,\sigma) =\varnothing$ and
otherwise it is finite since $X^d$ is compact.
\par
The \emph{sofic topological entropy} of $(X,G)$ with respect to $\Sigma$ is then defined as
\[
h_\Sigma(X,G) \coloneqq \sup_{\varepsilon > 0} \ \inf_{F \in \FF(G)} \ \inf_{\delta > 0} \ \limsup_{j \to \infty} \frac{1}{d_j} \log N_\varepsilon(\Map(X,\rho,F,\delta,\sigma_j), \rho_\infty).
\]
It can be shown that $h_{\Sigma}(X,G) \in \{-\infty\} \cup [0,\infty]$ does not depend on the choice of the metric
$\rho$ for $(X,G)$ (see \cite[Proposition~10.25]{kerr-li-book}).
It follows in particular that $h_\Sigma(X,G)$ is a \emph{topological conjugacy invariant}, i.e.,
if $(Y,G)$ is another dynamical system such that there exists a $G$-equivariant homeomorphism
$f \colon X \to Y$ then $h_\Sigma(X,G) = h_\Sigma(Y,G)$.

The following is a generalization of a result by Keyne and Robertson \cite{KR} for expansive homeomorphisms
(see also \cite[Chapter 16]{denker-grillenberger-sigmund} and, for a uniform version over amenable groups,
\cite[Corollary 4.21]{csc:exp-spec-top-ent-Myhill}).

\begin{lemma} \label{L-entropy at constant}
Let $G$ be a countable sofic group and let $(X,G)$ be an expansive dynamical system.
Fix a compatible metric $\rho$ on $X$ and let $c > 0$ be an expansivity constant for $(X,G)$ relative to $\rho$.
Let $\Sigma = (d_j, \sigma_j)_{j \in \N}$ be a sofic approximation for $G$. Then
\begin{equation}
\label{e:entropy-c}
h_\Sigma(X,G) = \inf_{F \in \cF(G)} \ \inf_{\delta>0} \ \limsup_{j\to \infty}\frac{1}{d_j}\log N_{c/2}(\Map(X,\rho, F, \delta,\sigma_j), \rho_\infty)
< +\infty.
\end{equation}
\end{lemma}
\begin{proof} Clearly $h_\Sigma(X,G)$ is no less than the right hand side in \eqref{e:entropy-c}.
By virtue of \cite[Proposition 10.23]{kerr-li-book}, it suffices to show that for any $\varepsilon>0$ one has
\begin{multline*}
\inf_{F\in \cF(G)} \ \inf_{\delta>0} \ \limsup_{j\to \infty}\frac{1}{d_j}\log N_{\varepsilon}(\Map(X,\rho, F, \delta,\sigma_j), \rho_2)\\
\leq \inf_{F\in \cF(G)} \ \inf_{\delta>0} \ \limsup_{j\to \infty}\frac{1}{d_j}\log N_{c/2}(\Map(X,\rho, F, \delta,\sigma_j), \rho_\infty).
\end{multline*}

Let $\varepsilon > 0$. By Lemma \ref{l:uniform-exp} we can find $K\in \cF(G)$ such that for any $x, y\in X$ with $\rho_K(x, y)\leq c$ one has $\rho(x, y)<\varepsilon/4$.

Let $F\in \cF(G)$ containing $K$ and let $\delta>0$ be small enough so that $\delta<(c/4)^2$ and
$2|K|\delta \diam(X, \rho)^2<(\varepsilon/4)^2$.
It suffices to show that for any $d\in \N$ and any map $\sigma \colon G \to \Sym(d)$ one has
\[
N_{\varepsilon}(\Map(X, \rho, F, \delta,\sigma), \rho_2)\leq N_{c/2}(\Map(X, \rho, F, \delta,\sigma), \rho_\infty).
\]

For each $\varphi\in \Map(X, \rho, F, \delta, \sigma)$ set
\[
W_\varphi \coloneqq \{a\in [d]: \rho(\varphi(\sigma_{s}a), s\varphi(a)) <\delta^{1/2} \mbox{ for all } s\in K\}.
\]
We claim that
\begin{equation}
\label{e:w-phi-big}
\frac{|W_{\varphi}|}{d} \geq 1-|K|\delta.
\end{equation}
In order to prove the claim, for $s \in K$ let us set $W_{\varphi,s} \coloneqq \{a\in [d]: \rho(\varphi(\sigma_sa), s\varphi(a)) < \delta^{1/2}\}$.
By applying the Chebysh\"ev inequality we have
\[
\begin{split}
\vert [d] \setminus W_{\varphi,s} \vert & = \vert \{a\in [d]: \rho(\varphi(\sigma_sa), s\varphi(a)) \geq \delta^{1/2} \} \vert\\
& \leq \frac{1}{\delta} \sum_{a \in  [d] \setminus W_{\varphi,s}} \rho(\varphi(\sigma_sa), s\varphi(a))^2 \\
& \leq \frac{1}{\delta} \sum_{a \in  [d]} \rho(\varphi(\sigma_sa), s\varphi(a))^2 \\
\mbox{(by \eqref{e:map-phi})} \ & \leq d \delta.
\end{split}
\]
We deduce that
\[
\vert [d] \setminus W_{\varphi} \vert  = \vert \bigcup_{s \in K} \left([d] \setminus  W_{\varphi,s}\right) \vert
\leq d |K| \delta,
\]
and the claim follows.

Take a $(\rho_\infty, c/2)$-separated subset $\Phi$ of $\Map(X, \rho, F, \delta, \sigma)$ with
\[
|\Phi| = N_{c/2}(\Map(X, \rho, F, \delta,\sigma), \rho_\infty).
\]
Let $\psi\in \Map(X, \rho, F, \delta, \sigma)$. Then $\rho_\infty(\psi, \varphi)<c/2$ for some $\varphi\in \Phi$.
For every $a\in W_\varphi\cap W_\psi$, we have
\begin{align*}
 \rho(s\varphi(a), s\psi(a))&\leq \rho(s\varphi(a), \varphi(\sigma_{s}a)) + \rho(\varphi(\sigma_{s}a), \psi(\sigma_{s}a)) +
\rho(\psi(\sigma_{s}a), s\psi(a))\\
 & \leq  \delta^{1/2}+c/2+\delta^{1/2}\leq c
 \end{align*}
 for all $s\in K$, and hence $\rho(\varphi(a), \psi(a)) < \varepsilon/4$ by our choice of $K$. Therefore
 \begin{align*}
 \rho_2(\varphi, \psi) & \leq \left(\frac{|W_\varphi\cap W_\psi|}{d}(\frac{\varepsilon}{4})^2 +
(1-\frac{|W_\varphi\cap W_\psi|}{d})\diam(X, \rho)^2\right)^{1/2}\\
 \mbox{(by \eqref{e:w-phi-big})} \ & \leq ((\frac{\varepsilon}{4})^2+2|K|\delta \diam(X, \rho)^2)^{1/2}< \varepsilon/2.
 \end{align*}
Then
\[
N_{\varepsilon}(\Map(X, \rho, F, \delta,\sigma), \rho_2)\leq |\Phi|=N_{c/2}(\Map(X, \rho, F, \delta,\sigma), \rho_\infty),
\]
as desired.
\end{proof}

\section{Proofs}
\begin{proof}[Proof of Proposition \ref{p:main}]
Take a compatible metric $\rho$ for $X$ and let $c>0$ be an expansivity constant for $(X,G)$ relative to $\rho$.
Take a point $x_0\in X\setminus Y$ and set $\varepsilon \coloneqq \min(\rho(x_0, Y), c)/6>0$.
Take a maximal $(\rho, \varepsilon)$-separated subset $Z$ of $X$.
\par
Since $(X,G)$ is expansive and has the strong TMP, by Remark \ref{r:exp-sTMP-usTMP} it has the uniform strong TMP.
Thus we can find $\tau>0$ and $F_1\in \cF(G)$ containing $1_G$ such that for any $A\in \cF(G)$, any
$F_1A$-separated $V\subset   G$, any $x_s\in X$ for $s\in V$, and any $y\in X$ satisfying $\rho_{(F_1A\setminus A)s}(x_s, y)\leq \tau$
for all $s\in V$, there is some $z\in X$ so that $\rho_{F_1As}(x_s, z)\leq \varepsilon$
for all $s\in V$ and $\rho_{G\setminus AV}(y, z)\leq \varepsilon$.
\par
Moreover, by the weak specification property, we can find some symmetric $F_2\in \cF(G)$ containing $1_G$ such that for any $A_1, A_2\subset   G$ satisfying $F_2A_1\cap F_2A_2=\varnothing$ and any $x_1, x_2\in X$ there is some $z\in X$ such that $\rho_{A_1}(x_1, z)\leq
\min(\tau/2, \varepsilon)$ and $\rho_{A_2}(x_2, z)\leq \min(\tau/2, \varepsilon)$. In particular (by taking $A_1 \coloneqq \{1_G\}$ and $A_2 \coloneqq G \setminus F_2^2$), for each $y\in Y$, there is some
$z_y \in X$ such that $\rho(x_0, z_y) \leq \min(\tau/2, \varepsilon)$ and $\rho_{G\setminus F_2^2}(y, z_y) \leq \min(\tau/2, \varepsilon)$.
\par
Let now $K\in \cF(G)$ and $\delta, \kappa>0$.

By Stirling's approximation formula (cf.\ Lemma \ref{l:stirling}), there is some $0<\gamma<1$ such that for any $d\in \Nb$, the number of subsets $D \subset   [d]$ with $|D|\leq \gamma d$ is at most $e^{\kappa d}$.
Take $0<\theta<\min(1/4, \gamma)$ such that
\begin{equation}
\label{e:stirling-z-e-kappa}
|Z|^{2\theta}\leq e^\kappa
\end{equation}
and $4\theta \, \diam(X, \rho)^2\leq (\delta/2)^2$.

Since $c$ is an expansivity constant for $(X,G)$, by Lemma \ref{l:uniform-exp} we can find some $K_1\in \cF(G)$ containing $1_G$ such that for any $x_1, x_2\in X$ with $\rho_{K_1}(x_1, x_2)\leq c$, one has $\rho(x_1, x_2)<\delta/2$.

By compactness of $X$, the continuous action of $G$ is indeed uniformly continuous. As a consequence, we can find $\xi>0$ such that
\begin{equation}
\label{e:rho}
\rho(x, y)\leq \xi \ \Rightarrow \rho_{F_1F_2^2}(x, y)\leq \tau/2 \ \mbox{ and } \ \rho_{K_1}(x, y)\leq \varepsilon
\end{equation}
for all $x, y\in X$.

Set $F \coloneqq F_1F_2^2F_1^{-1}K_1(\{1_G\}\cup K)\in \cF(G)$, and take $\zeta>0$ such that
\begin{equation}
\label{e:zeta}
\zeta\leq \xi^2 \ \mbox{ and } \ \zeta |F|\leq \theta.
\end{equation}

Let $j_0 \in \N$ be large enough so that for all $j \geq j_0$, setting for simplicity $d \coloneqq d_j$ and $\sigma \coloneqq \sigma_j$,
we have that (cf.\ Section \ref{s:sofic})
\begin{equation}
\label{e:U}
\frac{|U|}{d}\geq 1-\theta,
\end{equation}
where $U$ denotes the set of $a\in [d]$ satisfying that $\sigma_sa\neq \sigma_ta$ for all distinct $s, t\in F$ and $\sigma_{st}a = \sigma_s\sigma_ta$ for all $s, t\in F$. Note that $\sigma_{1_G}a = a$ for all $a\in U$, since $1_G \in F$.

For any $\varphi, \psi \colon [d]\rightarrow X$ and any $\Theta\subset   [d]$, we put
$$\rho_{\infty, \Theta}(\varphi, \psi) \coloneqq \max_{a\in \Theta}\rho(\varphi(a), \psi(a)).$$

As, by our assumptions, $h_\Sigma(X,G) \geq 0$, the statement would trivially hold if $h_\Sigma(Y,G) = - \infty$. Thus, we may assume that $\Map(Y,\rho,F,\zeta,\sigma) \neq \varnothing$. For $\varphi\in \Map(Y, \rho, F, \zeta, \sigma)$ we then set
\[
W_\varphi \coloneqq \{a\in [d]: \rho(\varphi(\sigma_sa), s\varphi(a))\leq \zeta^{1/2} \mbox{ for all } s\in F\}.
\]
As in the proof of Lemma \ref{L-entropy at constant}, we have (cf.\ \eqref{e:w-phi-big}) $\frac{|W_{\varphi}|}{d} \geq 1 -|F| \zeta$
so that as $|F|\zeta \leq \theta$ (cf.\ \eqref{e:zeta}),
\begin{equation}
\label{e:w-phi-big-theta}
\frac{|W_{\varphi}|}{d} \geq 1-\theta.
\end{equation}

Take a $(\rho_\infty, c/2)$-separated subset $\Phi$ of $\Map(Y, \rho, F, \zeta, \sigma)$ such that
\begin{equation}
\label{e:phi}
|\Phi|=N_{c/2}(\Map(Y, \rho, F, \zeta, \sigma), \rho_\infty).
\end{equation}

By our choice of $\theta$ and $\gamma$ (cf.\ \eqref{e:stirling-z-e-kappa}) there is some $\Phi_1\subset   \Phi$ such that
\begin{equation}
\label{e:phi-1}
|\Phi_1|\geq |\Phi|e^{-\kappa d}
\end{equation}
and the set $W_\varphi$ is the same for all $\varphi\in \Phi_1$.
Set $W \coloneqq W_\varphi$ for $\varphi\in \Phi_1$, and $U' \coloneqq U\cap W$.
Then, keeping in mind \eqref{e:U} and \eqref{e:w-phi-big-theta}, we have $\frac{|U'|}{d}\geq 1-2\theta$.

Since $Z$ is a maximal $(\rho, \varepsilon)$-separated subset of $X$, for every $\varphi\in \Phi_1$ and $a \in [d]\setminus U'$
we can find an element $f_\varphi(a) \in Z$ such that $\rho(\varphi(a), f_\varphi(a)) < \varepsilon$. This defines a map
$f_\varphi \colon [d]\setminus U'\rightarrow Z$.
Then, there is a set $\Phi_2\subset   \Phi_1$ such that
\[
|\Phi_2| \geq |\Phi_1| \cdot |Z|^{-|[d]\setminus U'|} \geq |\Phi_1|\cdot |Z|^{-2\theta d}\geq |\Phi_1|e^{-\kappa d}
\]
and the map $f_\varphi$ is the same for all $\varphi\in \Phi_2$.
Note that by \eqref{e:phi-1} we have
\begin{equation}
\label{e:phi-2}
|\Phi_2| \geq |\Phi|e^{-2\kappa d}.
\end{equation}

Now, for any distinct $\varphi, \psi\in \Phi_2$, we have $\rho_{\infty, [d]\setminus U'}(\varphi, \psi)< 2\varepsilon\leq c/2$ and $\rho_\infty(\varphi, \psi)\geq c/2$, and hence

\begin{equation}
\label{e:infty-U'-phi-psi}
\rho_{\infty, U'}(\varphi, \psi)\geq c/2.
\end{equation}

Take a maximal subset $\Lambda$ of $U'$ subject to the condition that $\sigma(F_1F_2^2)w_1\cap \sigma(F_1F_2^2)w_2=\varnothing$
for all distinct $w_1, w_2\in \Lambda$. Then $\sigma(F_1F_2^2)^{-1}\sigma(F_1F_2^2)\Lambda\supseteq U'$, and hence (recalling that
$0 < \theta \leq 1/4$)
\begin{equation}
\label{e:lambda}
\frac{|\Lambda|}{d} \geq \frac{|U'|}{|F_1F_2^2|^2d} \geq \frac{1-2\theta}{|F_1F_2^2|^2} \geq \frac{1}{2|F_1F_2^2|^2}.
\end{equation}

Let $B\subset   \Lambda$. Again, since $Z$ is a maximal $(\rho, \varepsilon)$-separated subset of $X$, for every
$\varphi\in \Phi_2$ and $a \in \sigma(F_1F_2^2)B$ we can find an element $f_{B, \varphi}(a) \in Z$ such that
$\rho(\varphi(a), f_{B, \varphi}(a)) < \varepsilon$. This defines a map
$f_{B, \varphi} \colon \sigma(F_1F_2^2)B\rightarrow Z$.
Then there is a subset $\Phi_B \subset   \Phi_2$ such that
\[
|\Phi_B| \geq |\Phi_2| \cdot |Z|^{-|\sigma(F_1F_2^2)B|} = |\Phi_2|\cdot |Z|^{-|F_1F_2^2|\cdot |B|}
\]
and the map $f_{B, \varphi}$ is the same for all $\varphi\in \Phi_B$.
Note that by \eqref{e:phi-2} we have
\begin{equation}
\label{e:phi-B}
|\Phi_B| \geq |\Phi| e^{-2\kappa d} |Z|^{-|F_1F_2^2| \cdot |B|}.
\end{equation}

For any distinct $\varphi, \psi\in \Phi_B$, we have $\rho_{\infty, \sigma(F_1F_2^2)B}(\varphi, \psi)< 2\varepsilon\leq c/2$ and $\rho_{\infty, U'}(\varphi, \psi)\geq c/2$, and hence
\begin{equation}
\label{e:infty-U'-setminus-phi-psi}
\rho_{\infty, U'\setminus \sigma(F_1F_2^2)B}(\varphi, \psi)\geq c/2.
\end{equation}

Let $a\in U'$. Denote by $V_{a}$ the set of all $s\in F$ such that $\sigma_sa\in \Lambda$. For any $s\in V_{a}$, since $F_1F_2^2\subset   F$ and $a\in U$, we have $\sigma(F_1F_2^2s)a=\sigma(F_1F_2^2)\sigma_s a$. Since $a \in U$, for any distinct $s, t$ in $V_{a}$, we have $\sigma_s a\neq \sigma_t a$, and hence
\[
\sigma(F_1F_2^2s)a\cap \sigma(F_1F_2^2t)a=\sigma(F_1F_2^2)\sigma_s a\cap \sigma(F_1F_2^2)\sigma_t a=\varnothing.
\]
It follows that $V_{a}$ is $F_1F_2^2$-separated.

With a subset $B\subset   \Lambda$ and $\varphi\in \Phi_B$ we associate a map $\bar{\varphi}_B \colon [d]\rightarrow X$ defined as follows.

Let $a \in [d]$. Suppose first that $a \in U'$.
Denote by $V_{B, a}$ the set of all $s\in F$ such that $\sigma_s a \in B$. Then $V_{B, a}\subset   V_a$ and hence $V_{B, a}$ is
$F_1F_2^2$-separated.
Let $s \in V_{B, a}$ and consider $\varphi(\sigma_s a) \in Y$.
By our choice of $F_2$, there is a point $z_{\varphi(\sigma_s a)} \in X$ with 
$\rho(x_0, z_{\varphi(\sigma_s a)}) \leq \min(\tau/2, \varepsilon)$ and
$\rho_{G\setminus F_2^2}(z_{\varphi(\sigma_s a)}, \varphi(\sigma_s a)) \leq \min(\tau/2, \varepsilon)$. 
By our choice of $\xi$, since $\rho(\varphi(\sigma_s a), s\varphi(a))\leq \zeta^{1/2}\leq \xi$, we have $\rho_{F_1F_2^2}(\varphi(\sigma_s a), s\varphi(a)) \leq \tau/2$.
Setting $x_{\varphi, a, s} \coloneqq s^{-1}z_{\varphi(\sigma_s a)} \in X$, we have

\begin{align*}
 \rho_{(F_1F_2^2\setminus F_2^2)s}(x_{\varphi, a, s}, \varphi(a))&\leq \rho_{(F_1F_2^2\setminus F_2^2)s}(s^{-1}z_{\varphi(\sigma_s a)}, s^{-1}\varphi(\sigma_s a))+\rho_{(F_1F_2^2\setminus F_2^2)s}(s^{-1}\varphi(\sigma_s a), \varphi(a))\\
 &=\rho_{F_1F_2^2\setminus F_2^2}(z_{\varphi(\sigma_s a)}, \varphi(\sigma_s a))+\rho_{F_1F_2^2\setminus F_2^2}(\varphi(\sigma_s a), s\varphi(a))\\
 &\leq \tau/2+\tau/2=\tau.
 \end{align*}
By our choice of $F_1$ and $\tau$ (coming from the uniform sTMP of $(X,G)$), there is some $\bar{\varphi}_B(a)\in X$ such that
\begin{equation}
\label{e:x-phi-a-s:bar-phi-B}
\rho_{F_1F_2^2s}(x_{\varphi, a, s}, \bar{\varphi}_B(a))\leq \varepsilon \ \mbox{ for all } \ s\in V_{B, a}
\end{equation}
and
\begin{equation}
\label{e:G-minus-f-2-V-b-a}
\rho_{G\setminus F_2^2V_{B, a}}(\varphi(a), \bar{\varphi}_B(a)) \leq \varepsilon.
\end{equation}

For $a\in [d]\setminus U'$, we set $\bar{\varphi}_B(a) \coloneqq x_0$.\\

\noindent
{\bf Claim 1.} $\bar{\varphi}_B\in \Map(X, \rho, K, \delta, \sigma)$.
\begin{proof}[Proof of Claim 1.]
Let $t\in K$, let $a\in U'\cap \sigma_t^{-1}U'$, and let $g\in K_1$.

Suppose $gt\in F_1F_2^2V_{B, a}$, say $gt=hs$ with $h\in F_1F_2^2$ and $s\in V_{B, a}$. Then $h^{-1}g\in F_2^2F_1^{-1}K_1 \subset   F$ and
$\sigma_{h^{-1}g}\sigma_ta=\sigma_{h^{-1}gt}a=\sigma_s a\in B$,  and hence $h^{-1}g\in V_{B, \sigma_t a}$. In this case we have
\begin{equation}
\label{e:gx}
gx_{\varphi, \sigma_t a, h^{-1}g} = g(h^{-1}g)^{-1}z_{\varphi(\sigma_{h^{-1}g}\sigma_t a)} =
hz_{\varphi(\sigma_{h^{-1}g}\sigma_t a)} = gts^{-1}z_{\varphi(\sigma_s  a)} = gtx_{\varphi, a, s},
\end{equation}
and hence using $g = h(h^{-1}g) \in F_1F_2^2 V_{B, \sigma_t a}$, we have
\begin{align*}
\rho(g\bar{\varphi}_B(\sigma_t a), gt\bar{\varphi}_B(a))
& \leq \rho(g\bar{\varphi}_B(\sigma_t a), gx_{\varphi, \sigma_t a, h^{-1}g})+\rho(gx_{\varphi, \sigma_t a, h^{-1}g}, gtx_{\varphi,  a, s})\\
& \hspace{5cm}+\rho(gtx_{\varphi, a, s}, gt\bar{\varphi}_B(a))\\
\mbox{(by \eqref{e:gx})} \ & = \rho(g\bar{\varphi}_B(\sigma_t a), gx_{\varphi, \sigma_t a, h^{-1}g})+\rho(gtx_{\varphi, a, s}, gt\bar{\varphi}_B(a))\\
\mbox{(as $gt=hs$)} \ & = \rho(g\bar{\varphi}_B(\sigma_t a), gx_{\varphi, \sigma_t a, h^{-1}g})+ 
\rho(hsx_{\varphi,a,s}, hs \bar{\varphi}_B(a))\\
\mbox{(by \eqref{e:x-phi-a-s:bar-phi-B})} \ & \leq \varepsilon + \varepsilon \leq c.
 \end{align*}

Suppose $g\in F_1F_2^2V_{B, \sigma_t a}$, say $g=h_1s_1$ with $h_1\in F_1F_2^2$ and $s_1\in V_{B, \sigma_t a}$. Then
$h_1^{-1}gt \in F_2^2 F_1^{-1}K_1 K \subset   F$ and $\sigma_{h_1^{-1}gt}a = \sigma_{h_1^{-1}g}\sigma_t a = \sigma_{s_1}\sigma_t a\in B$, and hence $h_1^{-1}gt\in V_{B, a}$.
In this case we have $gt=h_1(h_1^{-1}gt)\in F_1F_2^2V_{B, a}$, and hence, by the above, we get
$\rho(g\bar{\varphi}_B(\sigma_t a), gt\bar{\varphi}_B(a)) \leq c$.

Finally, suppose that $gt \not\in F_1F_2^2V_{B, a}$ and $g \not\in F_1F_2^2V_{B, \sigma_t a}$.
Since $\rho(\varphi(\sigma_t a), t\varphi(a))\leq \zeta^{1/2}\leq \xi$, and recalling that $g \in K_1$, we have
$\rho(g\varphi(\sigma_t a), gt\varphi(a))\leq \varepsilon$ by \eqref{e:rho}. As a consequence,
\begin{align*}
 \rho(g\bar{\varphi}_B(\sigma_t a), gt\bar{\varphi}_B(a))
 & \leq \rho(g\bar{\varphi}_B(\sigma_t a), g\varphi(\sigma_t a)) + \rho(g\varphi(\sigma_t a), gt\varphi(a))+\rho(gt\varphi(a), gt\bar{\varphi}_B(a))\\
 \mbox{(by \eqref{e:G-minus-f-2-V-b-a})} \ &\leq \varepsilon+ \varepsilon+\varepsilon\leq c.
 \end{align*}

 We conclude that $\rho_{K_1}(\bar{\varphi}_B(\sigma_t a), t\bar{\varphi}_B(a))\leq c$, and hence $\rho(\bar{\varphi}_B(\sigma_t a), t\bar{\varphi}_B(a))<\delta/2$ by our choice of $K_1$. Now we get
\[
\begin{split}
 \rho_2(\bar{\varphi}_B\circ \sigma_t, t\bar{\varphi}_B) & \leq ((\frac{\delta}{2})^2+(1-\frac{|U'\cap \sigma_t^{-1} U'|}{d}) \,\diam(X, \rho)^2)^{1/2}\\
& \leq ((\frac{\delta}{2})^2+4\theta \,\diam(X, \rho)^2)^{1/2}\\
& \leq \delta.
 \end{split}
\]
This proves our claim.
\end{proof}

Let $B\subset   \Lambda$ and $\varphi\in \Phi_B$.
For any $a\in B$, we have $1_G \in V_{B, a}$, and hence
\begin{equation}
\label{e:rho-bar-phi-B-x-0}
\rho(\bar{\varphi}_B(a), x_0)\leq \rho(\bar{\varphi}_B(a), x_{\varphi, a, 1_G}) + \rho(x_{\varphi, a, 1_G}, x_0) \leq \varepsilon +
\rho(z_{\varphi(a)}, x_0) \leq \varepsilon + \varepsilon =2\varepsilon.
\end{equation}
For any $a\in U'\setminus \sigma(F_1F_2^2)B$, we have $1_G \not \in F_1F_2^2V_{B, a}$, and hence
\begin{equation}
\label{e:rho-bar-phi-B-a}
 \rho(\bar{\varphi}_B(a), \varphi(a))\leq \varepsilon.
\end{equation}
In particular, \eqref{e:rho-bar-phi-B-a} holds for all $a \in \Lambda \setminus B \subset   U'\setminus \sigma(F_1F_2^2)B$.

Let us set $\overline{\Phi} \coloneqq \{\bar{\varphi}_B: B\subset   \Lambda \mbox{ and } \varphi\in \Phi_B\} \subset   \Map(X, \rho, K, \delta, \sigma)$.\\

\noindent
{\bf Claim 2.} {\it $\overline{\Phi}$ is $(\rho_\infty, \varepsilon)$-separated.}

\begin{proof}[Proof of Claim 2.]
For any distinct $B_1, B_2\subset   \Lambda$ and $\varphi\in \Phi_{B_1}, \psi\in \Phi_{B_2}$, we have for, say, $a\in B_1\setminus B_2$,
\begin{align*}
\rho_\infty(\bar{\varphi}_{B_1}, \bar{\psi}_{B_2}) & \geq \rho(\bar{\varphi}_{B_1}(a), \bar{\psi}_{B_2}(a))\\
& \geq \rho(x_0, \psi(a)) - \rho(\bar{\varphi}_{B_1}(a), x_0) - \rho(\bar{\psi}_{B_2}(a), \psi(a))\\
\mbox{(by \eqref{e:rho-bar-phi-B-x-0}, \eqref{e:rho-bar-phi-B-a}, and $\varepsilon \leq \rho(x_0,Y)/6$)} \   & \geq 4\varepsilon - 2\varepsilon -\varepsilon =\varepsilon.
 \end{align*}

For any $B\subset   \Lambda$ and any distinct $\varphi, \psi\in \Phi_B$, by \eqref{e:infty-U'-setminus-phi-psi} we have $\rho(\varphi(a), \psi(a)) \geq c/2$ for some
$a\in U'\setminus \sigma(F_1F_2^2)B$, and hence
\begin{align*}
\rho_\infty(\bar{\varphi}_B, \bar{\psi}_B) & \geq \rho(\bar{\varphi}_B(a), \bar{\psi}_B(a))\\
 & \geq \rho(\varphi(a), \psi(a)) - \rho(\varphi(a),\bar{\varphi}_B(a)) - \rho(\psi(a), \bar{\psi}_B(a))\\
\mbox{(by \eqref{e:rho-bar-phi-B-a})} \ & \geq c/2 - \varepsilon - \varepsilon \geq \varepsilon.
 \end{align*}
 This proves our claim.
\end{proof}

We need the following elementary fact. Given a finite set $\Lambda$ and a variable $t$, we have
\begin{equation}
\label{e:binom-lambda}
\sum_{B\subset   \Lambda} t^{|B|} = \sum_{j=0}^{|\Lambda|} {{|\Lambda|}\choose{j}} t^{j} = (1+t)^{|\Lambda|}.
\end{equation}

We are now in a position to prove the entropic inequality and conclude the proof of the proposition.

\begin{align*}
 N_\varepsilon(\Map(X, \rho, K, \delta, \sigma), \rho_\infty) & \geq \vert \overline{\Phi} \vert \geq \sum_{B\subset   \Lambda}|\Phi_B|\\
\mbox{(by \eqref{e:phi-B})}  & \geq \sum_{B\subset   \Lambda} |\Phi|e^{-2\kappa d} |Z|^{-|F_1F_2^2|\cdot |B|}\\
\mbox{(by \eqref{e:phi} and \eqref{e:binom-lambda})} \ & = N_{c/2}(\Map(Y, \rho, F, \zeta, \sigma), \rho_\infty)e^{-2\kappa d}(1+|Z|^{-|F_1F_2^2|})^{|\Lambda|}\\
 \mbox{(by \eqref{e:lambda})} & \geq N_{c/2}(\Map(Y, \rho, F, \zeta, \sigma), \rho_\infty)e^{-2\kappa d}
(1+|Z|^{-|F_1F_2^2|})^{d/(2|F_1F_2^2|^2)}.
 \end{align*}
It follows that
\begin{align*}
& \limsup_{i\to \infty}\frac{1}{d_i}\log N_\varepsilon(\Map(X, \rho, K, \delta, \sigma_i), \rho_\infty)\\
&\geq \limsup_{i\to \infty}\frac{1}{d_i}\log N_{c/2}(\Map(Y, \rho, F, \zeta, \sigma_i), \rho_\infty)-2\kappa+\frac{1}{2|F_1F_2^2|^2}\log (1+|Z|^{-|F_1F_2^2|})\\
&\geq h_\Sigma(Y,G)-2\kappa+\frac{1}{2|F_1F_2^2|^2}\log (1+|Z|^{-|F_1F_2^2|}),
\end{align*}
where in the 2nd inequality we apply Lemma~\ref{L-entropy at constant} to $(Y,G)$. Letting $\kappa\to 0$, we obtain
$$\limsup_{i\to \infty}\frac{1}{d_i}\log N_\varepsilon(\Map(X, \rho, K, \delta, \sigma_i), \rho_\infty)\geq h_\Sigma(Y,G)+\frac{1}{2|F_1F_2^2|^2}\log (1+|Z|^{-|F_1F_2^2|}).$$
Taking the infimum over $K$ and $\delta$, we get
$$ h_\Sigma(X,G)\geq h_\Sigma(Y,G)+\frac{1}{2|F_1F_2^2|^2}\log (1+|Z|^{-|F_1F_2^2|}) > h_\Sigma(Y,G).$$
\end{proof}

\begin{proof}[Proof of Theorem \ref{t:main}]
Let $f \colon X \to X$ be an injective endomorphism of the dynamical system $(X,G)$.
By continuity of $f$ and since $X$ is compact and Hausdorff, $f$ yields a homeomorphism between $X$ and $Y \coloneqq f(X)$.
Moreover, since $f$ is $G$-equivariant, $Y \subset   X$ is $G$-invariant.
It follows that $f$ in fact establishes a conjugacy between the dynamical systems $(X,G)$ and $(Y,G)$ so that
$h_\Sigma(X,G) = h_\Sigma(Y,G)$.
By Proposition \ref{p:main}, $Y$ cannot be a proper subset of $X$. Thus $f(X) = Y = X$, that is, $f$ is surjective.
This shows that $(X,G)$ is surjunctive.
\end{proof}

\begin{proof}[Proof of Corollary \ref{c:alg-finitely-presented}]
First of all, Deninger and Schmidt \cite[Theorem 3.2]{DS} proved that a principal algebraic dynamical system $(X_f,G)$ is expansive if and only if $f$ is invertible in $\ell^1(G)$ and Ren \cite[Theorem 1.2]{Ren} proved that every expansive principal algebraic action has the weak specification property. The proofs of the analogous statements for $A \in \M_n(\Z[G])$, $n \geq 1$, follow the same lines.
Secondly, Meyerovitch \cite[Theorem 3.4]{Meyerovitch} proved that the finitely-presented dynamical system $(X_A,G)$ associated with an element $A \in \M_{n}(\Z[G])$ which is invertible in $\M_n(\ell^1(G))$ has the pseudo-orbit tracing property.
Finally, every dynamical system with the POTP has the sTMP \cite[Proposition 3.11]{BGL}.
The statement then follows from Theorem \ref{t:main}.
\end{proof}

\begin{proof}[Proof of Theorem \ref{T-CPE to surjunctive}]
It is a result of Deninger that for any action of a countably infinite amenable group $G$ on a compact metrizable group $X$ by continuous automorphisms, the topological entropy is equal to the measure entropy $h_{\mu_X}(X, G)$ of the normalized Haar measure $\mu_X$ 
\cite[Theorem 2.2]{Deninger} (though Deninger assumed $X$ to be abelian, the proof there does not use this assumption). 
It was shown in \cite[Corollaries 7.5 and 8.4]{chung-li-2015} that such an action has CPE if and only if it has CPE with respect to $\mu_X$, i.e., every nontrivial factor of the probability-measure-preserving action $G\curvearrowright (X,\mu_X)$ has positive entropy. 
Furthermore, it was shown in \cite[Theorem 8.6]{chung-li-2015} that such an action with finite measure entropy of $\mu_X$ has CPE with respect to $\mu_X$ if and only if, for every $G$-invariant Borel probability measure $\nu$ of $X$ not equal to $\mu_X$, one has 
$h_{\nu}(X,G) < h_{\mu_X}(X,G)$. 
\par
Let now $f \colon X \to X$ by an injective endomorphism of $(X,G)$.
We have $h_{f_*\mu_X}(X,G) = h_{f_*\mu_X}(f(X),G) = h_{\mu_X}(X,G)$ so that $f_*\mu_X = \mu_X$. 
Since $\mu_X$ has full support, we deduce that $f(X)=X$, that is, $f$ is surjective.
This shows that the dynamical system $(X,G)$ is surjucntive.
\end{proof}

\section{Examples}
\begin{example}[The hard-ball model]
Take $G = \Z$ and $A = \{0,1\}$.
Then the  subset $X \subset   A^G$, consisting of all $x \colon \Z \to \{0,1\}$ such that
$(x(n),x(n + 1)) \not= (1,1)$ for all $n \in \Z$, is a subshift.
This subshift is called the \emph{golden mean subshift}.
\par
More generally, take $G =\Z^d$ and $A = \{0,1\}$.
Then the  subset $X \subset   A^G$, consisting of all $x \in A^G$ such that
$(x(g),x(g + e_i)) \not= (1,1)$ for all $g \in G$, where $(e_i)_{1\leq i \leq d}$ is the canonical basis of $\Z^d$,
is a subshift.
This subshift is called the \emph{hard-ball model}.
\par
The hard-ball model is strongly irreducible and of finite type (and therefore splicable). To see this it suffices to convince ourselves that
the set $\Delta \coloneqq \{g = (g_1, \ldots, g_d) \in G: |g_i| \leq 1, i=1,\ldots,d\}$ witnesses both strong irreducibility
and the finite type condition.
\par
In \cite[Example 6.6]{BGL} the following generalization is presented. Let $G$ be a group,
let $F \subset   G$ be a finite nonempty subset not containing $1_G$.
The $F$-{hard-ball model} of $G$ is the subshift $X_{{\small{\rm hard}},F} \subset   \{0,1\}^G$ defined by setting
\[
X_{{\small{\rm hard}},F} \coloneqq \{x \in \{0,1\}^G: x(g)x(gs) = 0 \mbox{ for all } g \in G \mbox{ and } s \in F\}.
\]
Then $X_{{\small{\rm hard}},F}$ is a strongly irreducible subshift of finite type, and therefore has the POTP and thus the sTMP, equivalently the usTMP (cf.\ \cite[Proposition 3.11]{BGL} and Proposition \ref{p:sTMP}).
When $G$ is sofic, given any sofic approximation $\Sigma$ for $G$, one has $h_\Sigma(X_{{\small{\rm hard}},F},G) \geq 0$, since
the action has fixed points (in fact, only one, namely the constant $0$-configuration).
\end{example}

\begin{example}[Weiss' example]
\label{ex:weiss}
Benjy Weiss \cite{weiss-sgds} found the following simple example of a shift of finite type that is not surjunctive.
Let $A = \{0,1,2\}$ and $G = Z$. Consider the subshift $X \subset   A^G$ of finite type with defining set of admissible patterns
${\mathcal A} = \{00, 01, 11, 12, 22\} \subset   A^2 = A^\Omega$, where $\Omega = \{0,1\}\subset   G$ (cf.\ Section \ref{s:shifts}).
Thus any point $x$ in $X$ has at most one block of $1$'s, which, if finite is bordered by an infinite string of $0$'s to the left
and $2$'s to the right. Set $S \coloneqq \{-1,0\} \subset   G$ and define $\mu \colon A^S = A^2 \to A$ by setting
\[
\mu(s,t) =
\begin{cases} t & \mbox{ if } (s,t) \neq (1,2)\\
              1 &  \mbox{ if } (s,t) = (1,2).
\end{cases}																								
\]
Then the associated cellular automaton $f \colon A^G \to A^G$ leaves $X$ invariant (and therefore is an
endomorphism of $(X,G)$) and works as follows: if the finite block of $1$'s occurs in $x \in X$,
then in $f(x)$ the same block is elongated by one extra $1$ on the right.
Thus $f$ is injective. Clearly, $f$ is not surjective since a configuration with a single $1$ will be in $X \setminus f(X)$.
As a consequence $(X,G)$ is not surjunctive.
\end{example}

\appendix
\section{Stirling's approximation formula}
\label{app}
The following old result was widely used, without proof, in \cite{kerr-li}.
For the reader's convenience and the sake of completeness, we include its detailed proof.
\par
Given a real number $\alpha$ we denote by $\floor{\alpha} \in \Z$ the greatest integer less than or equal to $\alpha$.

\begin{lemma}
\label{l:stirling}
Let $0 < \gamma < 1/2$. Then there exists $\kappa = \kappa(\gamma) > 0$ and $d_0 = d_0(\gamma) \in \N$ with
\begin{equation}
\label{e:stirling-2}
\lim_{\gamma \to 0} \kappa(\gamma) = 0
\end{equation}
such that
\begin{equation}
\label{e:stirling}
\sum_{j=0}^{\floor{\gamma d}} {{d}\choose{j}} \leq e^{\kappa d}
 \end{equation}
for all $d \in \N$ such that $d \geq d_0$.
\end{lemma}
\begin{proof}
Let $d \geq 2/\gamma$ so that, on the one hand
\begin{equation}
\label{e:correction}
\floor{\gamma d + 1} = \floor{\gamma d} + 1 \leq 2 \floor{\gamma d},
\end{equation}
and, on the other hand, since $\gamma < 1/2$, one has $d \geq 4$ and therefore $d - \gamma d \geq 2$.

As $\gamma < 1/2$, we clearly have
\[
\sum_{j=0}^{\floor{\gamma d}} {{d}\choose{j}} \leq (\floor{\gamma d} + 1) {{d}\choose{\floor{\gamma d}}},
\]
so that, by virtue of \eqref{e:correction}, we only need to find  $\kappa = \kappa(\gamma)$ satisfying \eqref{e:stirling-2} and such that
\begin{equation}
\label{e:to-prove}
2\floor{\gamma d} {{d}\choose{\floor{\gamma d}}} \leq e^{\kappa d}.
\end{equation}

Setting $\kappa \coloneqq -2(\gamma \ln \gamma + (1-\gamma)\ln(1-\gamma)) > 0$, we clearly have that \eqref{e:stirling-2} holds. Consider the function $G(x) = e^{\kappa x/2} - x^3$. It is easy to see that there exists an integer $d_0\ge 2/\gamma$ such that $G(x) > 0$
for all $x \ge d_0$.

In \cite[Lemma 10.1]{kerr-li-book} the following version of Stirling's classical formula is proved by purely elementary means:
\begin{equation}
\label{e:stirling-KL}
e \left(\frac{m}{e}\right)^m \leq m! \leq em\left(\frac{m}{e}\right)^m
\end{equation}
for all $m \in \N$ such that $m \geq 1$.

Let $d \geq d_0$. Then, setting $m \coloneqq \floor{\gamma d}$, we deduce that
\[
\begin{split}
2 \floor{\gamma d} {{d}\choose{\floor{\gamma d}}} & = 2m {{d}\choose{m}} = 2m \cdot \frac{d!}{m!(d-m)!} \\
& \leq 2m \cdot \frac{e d \left(\frac{d}{e}\right)^d}{e \left(\frac{m}{e}\right)^m \cdot e \left(\frac{d-m}{e}\right)^{d-m}}\\
& = 2m \cdot \frac{d}{e} \cdot d^d \cdot \frac{1}{m^m} \cdot \frac{1}{(d-m)^{d-m}}\\
\mbox{(since $\gamma d - 1\leq m \leq \gamma d$)} \ & \leq 2 \gamma d \cdot \frac{d}{e} \cdot d^d \cdot \frac{1}{(\gamma d -1)^{\gamma d -1}} \cdot \frac{1}{((1-\gamma)d)^{(1-\gamma)d}}\\
& = 2 \gamma d \cdot \frac{d}{e} \cdot d^d \cdot \gamma d \cdot
\left(\frac{\gamma d}{\gamma d -1}\right)^{\gamma d -1} \\
& \ \ \ \ \ \ \ \ \ \ \ \ \ \ \ \ \ \  \cdot \frac{1}{d^d} \cdot \frac{1}{\gamma^{\gamma d}} \cdot \frac{1}{(1-\gamma)^{(1-\gamma)d}}\\
\mbox{(since $(\gamma d/(\gamma d -1))^{\gamma d -1} \leq e$)} \  & \leq 2 d^3 \gamma^2 \cdot \frac{1}{\gamma^{\gamma d}} \cdot \frac{1}{(1-\gamma)^{(1-\gamma)d}}\\
\mbox{(since $\gamma < 1/2$)} \ & < d^3 \cdot  \left(\frac{1}{\gamma^\gamma} \cdot \frac{1}{(1-\gamma)^{(1-\gamma)}}\right)^d\\
& = d^3 \cdot e^{-(\gamma \ln \gamma + (1-\gamma)\ln(1-\gamma))d}\\
\mbox{(
as $d \geq d_0$)} \  & \leq e^{\kappa d}.
\end{split}
\]
This proves \eqref{e:to-prove} and the proof is complete.
\end{proof}


\end{document}